\documentclass[absolute]{mymathart}
\usepackage[theorems]{mymathmacros}
\usepackage[usenames,dvipsnames]{color}
\usepackage{subfig}
\usepackage[shortlabels]{enumitem}
\usepackage{hyperref}
\usepackage{amssymb}
\usepackage{graphicx}
\usepackage{epigraph}
\usepackage{amsfonts,amssymb,color}
\usepackage[mathscr]{eucal}
\usepackage{amsmath, amsthm, amscd}
\usepackage{mathrsfs}
\usepackage{amsbsy}
\usepackage{float}
\usepackage{verbatim}
\usepackage{amsfonts} 
\usepackage{pifont}

\title[Families of infinite parabolic IFS: the approximating method]{Families of infinite parabolic IFS with overlaps: the approximating method}

\author{Liangang Ma}

\address{Dept.\ of Mathematical Sciences, Binzhou University, Huanghe 5th Road No. 391, Binzhou 256600, Shandong, P. R. China} 
\email{maliangang000@163.com}
\thanks{The work is supported by ZR2019QA003 from SPNSF and 12001056 from NSFC}    

\newtheorem{theorem}[subsection]{Theorem}
\newtheorem{lemma}[subsection]{Lemma}
\newtheorem{Young's Lemma}[subsection]{Young's Lemma}
\newtheorem{MU Theorem}[subsection]{Mauldin-Urba\'nski Theorem}
\newtheorem{SSU Theorem}[subsection]{Simon-Solomyak-Urba\'nski Theorem}
\newtheorem{Prohorov's Theorem}[subsection]{Prohorov's Theorem}
\newtheorem{proposition}[subsection]{Proposition}
\newtheorem{corollary}[subsection]{Corollary}
\newtheorem{definition}[subsection]{Definition}
\newtheorem{exm}[subsection]{Example}
\newtheorem{rem}[subsection]{Remark}

\numberwithin{equation}{section}

\begin{document} 

\begin{abstract}
   
This work is devoted to the study of families of infinite parabolic iterated function systems (PIFS) on a closed interval parametrized by vectors in $\mathbb{R}^d$ with overlaps. We show that the Hausdorff dimension and absolute continuity of ergodic projections through a family of infinite PIFS are decided \emph{a.e.} by the growth rate of the entropy and Lyapunov exponents of the families of truncated PIFS with respect to the concentrating measures, under transversality of the family essentially. We also give an estimation on the upper bound of the Hausdorff dimension of parameters where the corresponding ergodic projections admit certain dimension drop. The setwise topology on the space of finite measures enables us to approximate a family of infinite systems by families of their finite truncated sub-systems, which plays the key role throughout our work.                
\end{abstract}
 
 \maketitle

\section{Introduction}\label{sec1}

This work should be understood in the background of families of overlapping iterated function systems.  A typical family of overlapping \emph{iterated function system} (IFS) on a compact metric space consists of a flow of finitely many strictly contractive endomorphisms on the space parametrized by some (time) parameter. The dimension of attractors and measures supported on them (especially the invariant or ergodic ones) is the focus in the theory of families of IFS. Besides calculating the dimension of attractors and measures supported on the attractors, another important problem is to decide whether the projective measures on the attractors from measures on the symbolic spaces are singular or absolutely continuous with respect to the Lebesgue measure, refer to \cite{BRS, Hoc1, Hoc3, PSS, Shm2, Shm3, SS1, SSS}. See also \cite{Fur} by Furstenberg, \cite{Hoc4} by Hochman, \cite{PS1} by Peres-Solomyak and \cite{Shm4} by Shmerkin for more interesting questions on families of overlapping IFS. These problems are already very difficult in the case of Bernoulli convolutions, see for example \cite{BV1, BV2, Hoc1, LPS, Shm1, Sol3, SS2, Var1, Var2, Var3} for recent progress on the topic.

Note that most of the above research are done on families of finite IFS, that is, every individual IFS in the family is constituted by finitely many maps. In this work we try to deal with families of infinite IFS. We choose the families of parabolic iterated function systems investigated by K. Simon, B. Solomyak and M. Urba\'nski in \cite{SSU1, SSU2} to demonstrate our ideas, which are probably applicable to some other families of infinite IFS. Our technique here is to utilize the sequence of concentrating measures on the attractors of families of finite sub-systems to approximate the measures on the attractors of families of infinite systems. 

Let $X\subset \mathbb{R}$ be a closed interval with its Borel $\sigma$-algebra $\mathcal{B}$. Let $\mathcal{\hat{M}}(X)$ be the collection of all the finite Borel measures on $(X, \mathcal{B})$. For a set $A\subset\mathbb{R}^d$ with $d\in\mathbb{N}$, let $HD(A)$ be its Hausdorff dimension \cite{Fal2}.

\begin{definition}\label{def3}
For a measure $\nu\in\mathcal{\hat{M}}(X)$,  the \emph{lower} and \emph{upper Hausdorff dimension} of $\nu$ are defined respectively to be:

\begin{center}
$dim_* \nu=\inf\{HD(A): \nu(A)>0, A\in \mathcal{B}\},$  
\end{center}
and
\begin{center}
$dim^* \nu=\inf\{HD(A): \nu(A)=\nu(X), A\in \mathcal{B}\}.$
\end{center}
\end{definition} 

The two values indicate the distribution of a measure on $X$ from two complementary points of views, with the possibility of the lower one being strictly less than the upper one. However, they coincide with each other for ergodic measures with respect to transformations of proper regularity (for example the $C^1$ diffeomorphisms) on $X$. See for example \cite{HS, SimS2,  You} on calculation of the Hausdorff dimension of measures in various circumstances. Following P. Mattila, M. Mor\'an and J. M. Rey \cite{MMR}, the author discovers some semi-continuity of the measure-dimension mappings $dim_*$ and $dim^*$ on $\mathcal{\hat{M}}(X)$ equipped with the setwise topology \cite{FKZ, Las}, which is crucial in our approximating process.     
     
Now we gradually introduce the families of parabolic iterated function systems. For $\theta\in (0,1]$,  a point $v\in X$ is called an \emph{indifferent point} of a $C^{1+\theta}$ map $s: X\rightarrow X$ if 
\begin{center}
$|s'(v)|=1$.
\end{center}

\begin{definition}
A $C^{1+\theta}$ map $s: X\rightarrow X$ is called \emph{parabolic} it satisfies the following conditions:
\begin{itemize}
\item $s$ has only one indifferent point $v$ and it is fixed, that is, $s(v)=v$.
\item $s$ is \emph{contractive} on $X\setminus\{v\}$, that is, $0<|s'(x)|<1$ for any $x\in X\setminus\{v\}$. 
\item $s$ is \emph{well-behaved} around $v$, that is, $s'(x)$ is monotone on each component of $X\setminus\{v\}$.
\item There exist $L_1\geq 1$ and $\beta<\cfrac{\theta}{1-\theta}$ such that
\begin{center}
$\cfrac{1}{L_1}\leq\liminf_{x\rightarrow v}\cfrac{|s'(x)-s'(v)|}{|x-v|^\beta}\leq\limsup_{x\rightarrow v}\cfrac{|s'(x)-s'(v)|}{|x-v|^\beta}\leq L_1$.
\end{center}
\end{itemize} 
\end{definition}

It is called \emph{hyperbolic} if $0<|s'(x)|<1$ for any $x\in X$. Parabolic maps with finitely many indifferent points can be handled in our context, with some technical modifications on the proofs.

\begin{definition}\label{def8}
For $\theta\in (0,1]$ and a countable index set $I$ (with at least two elements), a \emph{parabolic iterated function system} (PIFS) is a collection of $C^{1+\theta}$ maps
\begin{center}
$S=\{s_i: X\rightarrow X\}_{i\in I}$
\end{center}
satisfying the following conditions:

\begin{itemize}
\item There is one and only one $i_1\in I$ such that the map $s_{i_1}$ is a parabolic map with its indifferent point $v$.\\

\item The other maps $\{s_i\}_{i\in I\setminus\{i_1\}}$ are all hyperbolic.

\end{itemize}
\end{definition}

The collection of PIFS $S=\{s_i: X\rightarrow X\}_{i\in I}$ satisfying
\begin{center}
$\cup_{i\in I\setminus\{i_1\}} s_i(X)\subset X^o\setminus\{v\}$
\end{center} 
is termed $\Gamma_X(\theta)$. For an integer $d\in\mathbb{N}$ and an open set $U\subset\mathbb{R}^d$, we focus on families of PIFS instead of a single system  in $\Gamma_X(\theta)$ parametrized by the vector-time parameter $\bold{t}\in U$. The parabolic map is always fixed at any time in the family. Since families of finite PIFS (that is, $\#I<\infty$) has been considered by Simon-Solomyak-Urba\'nski, we focus on the families of infinite PIFS ($\#I=\infty$) in the following. That is, we focus on families of PIFS 
\begin{center}
$S^{\bold{t}}=\{s_i^{\bold{t}}: X\rightarrow X\}_{i\in \mathbb{N}}$,
\end{center} 
such that  $S^{\bold{t}}\in\Gamma_X(\theta)$ with the parabolic map $s_1^{\bold{t}}=s_1$ remaining the same at any time $\bold{t}\in U$. Let 
\begin{center}
$\pi_{\bold{t}}: I^\infty\rightarrow X$
\end{center}
be the projection map and $J_{\bold{t}}=\pi_{\bold{t}}(I^\infty)$ be the attractor (\cite{Hut}) at time $\bold{t}\in U$. To achieve some reasonable conclusions on the parametrized family, obviously some control on the dependence of the family with respect to $\bold{t}$ is necessary. A family of PIFS $S^{\bold{t}}=\{s_i^{\bold{t}}: X\rightarrow X\}_{i\in \mathbb{N}}$ is said to satisfy the \emph{continuity condition} if for any fixed $i\in \mathbb{N}$, the map $s_i^{\bold{t}}$ depends continuously on the parameter $\bold{t}\in U$ in the Banach space $C^{1+\theta}$ equipped with the supreme norm. Let $\mathfrak{L}^d$ be the Lebesgue measure on $\mathbb{R}^d$ for $d\in\mathbb{N}$. The family is said to satisfy the \emph{transversality condition}  if there exists a constant $C_1$ such that
\begin{center}
$\mathfrak{L}^d\{\bold{t}\in U: |\pi_\bold{t}(\omega)-\pi_\bold{t}(\tau)|\leq r\}\leq C_1r$
\end{center}      
for any $\omega, \tau\in \mathbb{N}^\infty, \omega_1\neq \tau_1$ and any $r>0$.

While the continuity condition comes more naturally, the transversality condition has been recognized to be an effective condition to achieve some uniform conclusions in various contexts of flows of dynamical systems. For example, see  \cite{Fal1, PolS, PS2, PS3, SimS1, Sol1, Sol2}.  

The following are some classical notions on the symbolic dynamics on $\mathbb{N}^\infty$. For an infinite word $\omega=\omega_1 \omega_2 \cdots\in \mathbb{N}^\infty$, let
\begin{center}
$\sigma(\omega)=\omega_2 \omega_3 \cdots$
\end{center}
be the shift map. For a finite $k$-word $\tau\in \mathbb{N}^k$, let 
\begin{center}
$[\tau]=\{\omega\in \mathbb{N}^\infty: \omega|_k=\tau\}$
\end{center} 
be a \emph{cylinder set}, in which $\omega|_k=\omega_1\omega_2\cdots\omega_k$ is the $k$-th restriction of the infinite word $\omega$. Let
\begin{center}
$\mathbb{N}^*=\cup_{k=1}^\infty \mathbb{N}^k$
\end{center}
be the collection of all the finite words. Now consider the $\sigma$-algebra $\mathcal{B}_{\mathbb{N}^\infty}$ generated by all the cylinder sets on $\mathbb{N}^\infty$.  Let $h_\mu(\sigma)$ be the entropy of the shift map $\sigma$  with respect to the partition of $\mathbb{N}^\infty$ by cylinder sets $\{[i]\}_{i=1}^\infty$ for a measure $\mu\in\mathcal{\hat{M}}(\mathbb{N}^\infty)$ (\cite{Hoc2, Wal}).

One can treat $\mu$ as the law of an infinite discrete-time \emph{stochastic process} 
\begin{center}
$\bold{Y}=\{Y_i\}_{i\in\mathbb{N}}$,
\end{center}
in which $Y_i$ is a discrete random variable on $\mathbb{N}$ for $i\in\mathbb{N}$ under the natural projection with law $\mu(\mathbb{N}\times \mathbb{N}\times\cdots\times Y_i\times \mathbb{N}\times\cdots)$. The sequence of infinite random variables is called \emph{independent} if the finite variables 
\begin{center}
$\{Y_1, Y_2, \cdots, Y_n\}$
\end{center}
are independent for any $n\in\mathbb{N}$, see \cite[Definition 3a(1),(b)]{Tsi}. The stochastic process $\bold{Y}$ is a \emph{Markov process} under the hypothesis of independence, refer to \cite{BHP, FKZ, Hai, HL}. 

In the following we focus on the measures on $J_{\bold{t}}$ projected from probability measures on $\mathbb{N}^\infty$ through $\pi_{\bold{t}}$ for $\bold{t}\in U$.  For a measure $\mu$ on the symbolic space $\mathbb{N}^\infty$, consider its projection under $\pi_{\bold{t}}$:
\begin{center}
$\nu_{\bold{t}}=\mu\circ\pi_{\bold{t}}^{-1}$
\end{center}
for any $\bold{t}\in U$. We are particularly interested in the cases when $\mu$ is invariant or ergodic with respect to the shift map $\sigma$ on $\mathbb{N}^\infty$. If $\mu$ is invariant then  $\nu_{\bold{t}}$ is of pure type, see for example \cite{JW}. Let
\begin{center}
$\lambda^\bold{t}_\mu(\sigma)=-\int_{\mathbb{N}^\infty} \log |s_{\omega_1}'(\pi_{\bold{t}}\circ\sigma(\omega))|d\mu(\omega)$ 
\end{center}
be the average \emph{Lyapunov exponent} of the family of PIFS $S^{\bold{t}}=\{s_i^{\bold{t}}: X\rightarrow X\}_{i\in \mathbb{N}}$ at time $\bold{t}\in U$. 

Our first main result deals with the two basic questions regarding the projective measures $\{\nu_{\bold{t}}\}_{\bold{t}\in U}$, these are, deciding their dimensions $dim_* \nu_{\bold{t}}, dim^* \nu_{\bold{t}}$ and whether they are absolutely continuous or singular with respect to the Lebesgue measure for $\bold{t}\in U$. Let $\mathbb{N}_n=\{1,2,\cdots,n\}$ be the $n$-th truncation of $\mathbb{N}$ for $n\in \mathbb{N}$. We leave the technical concept-the $n$-th concentrating measure $\mu_n$ (supported on $\mathbb{N}_n^\infty$) of a finite measure $\mu$ (supported on $\mathbb{N}^\infty$) to Section \ref{sec4}. The following result generalizes \cite[Theorem 2.3]{SSU1} from families of finite PIFS to families of infinite PIFS. 

\begin{theorem}\label{thm17}
Let 
\begin{center}
$\big\{S^\bold{t}=\{s_i^{\bold{t}}: X\rightarrow X\}_{i\in \mathbb{N}}\in \Gamma_X(\theta)\big\}_{\bold{t}\in U}$
\end{center}
be a family of infinite parabolic iterated function systems satisfying the continuity and transversality condition with respect to the vector-time parameter $\bold{t}\in U$.  For an ergodic probability measure $\mu$ on the symbolic space $\mathbb{N}^\infty$ with positive entropy $h_\mu(\sigma)$, let $\mu_n$ be its $n$-th concentrating measure for any $n\in\mathbb{N}$. If the sequence of infinite random variables in $\bold{Y}$ under the law $\mu$ is independent and $\mu(\cup_{n=1}^\infty \mathbb{N}_n^\infty)=1$, then 
\begin{enumerate}[(i).]
\item For Lebesgue \emph{a.e.} $\bold{t}\in U$, $dim_* \nu_\bold{t}=dim^* \nu_\bold{t}=\lim_{n\rightarrow\infty}\min\Big\{ \cfrac{h_{\mu_n}(\sigma)}{\lambda^{\bold{t}}_{\mu_n}(\sigma)},1\Big\}$,

\item $\nu_\bold{t}$ is absolutely continuous for Lebesgue \emph{a.e.} $\bold{t}\in \Big\{\bold{t}: \limsup_{n\rightarrow\infty}\big\{\cfrac{h_{\mu_n}(\sigma)}{\lambda^{\bold{t}}_{\mu_n}(\sigma)}\big\}>1\Big\}$,
\end{enumerate}
in which $\nu_\bold{t}=\mu\circ\pi_\bold{t}^{-1}$ is the  projective measure at time $\bold{t}\in U$.
\end{theorem} 

\begin{rem}\label{rem3}
We are not sure whether the the hypothesis of independence of the sequence of infinite random variables under the law $\mu$ can be removed or not, although it seems so alluring to us.  
\end{rem} 

\begin{rem}\label{rem2}
As to the approximation of $\mu$, better results are possible if one does the approximation by the sequence of restrictions $\{\mu|_{\mathbb{N}_n^\infty}\}_{n=1}^\infty$ instead of the sequence of concentrating measures $\{\mu_n\}_{n=1}^\infty$. However, the restrictions are not probabilities, while the concentrating measures take the advantage that they are all probability measures.
\end{rem} 

\begin{rem}
Similar to \cite[Theorem 2.3]{SSU1}, we can only guarantee the result holds \emph{a.e.} instead of everywhere in Theorem \ref{thm17}, as it is still impossible to rule out the exact overlapping between images of $\{s^{\bold{t}}_{\omega}\}_{\omega\in \mathbb{N}^*}$ for us for fixed $\bold{t}\in U$. There is an important conjecture that a dimension drop implies exact overlapping, which has been confirmed in the case of finite Bernoulli convolutions, see for example \cite{Hoc1, Var3}.
\end{rem} 

Note that the proof of \cite[Theorem 2.3]{SSU1} does not apply to Theorem \ref{thm17}. There are two main obstacles on application of Simon-Solomyak-Urba\'nski's techniques to the families of infinite PIFS. The first one is that  $h_\mu(\sigma)$ can explode in case of $\#I=\infty$, see \cite{Hoc2}.  If $h_\mu(\sigma)=\infty$ for some ergodic $\mu\in\mathcal{\hat{M}}(\mathbb{N}^\infty)$ with respect to the shift map, the Shannon-McMillan-Breiman Theorem is not known to hold in this case \cite{Hay}. The second one is the regularity of the families of infinite PIFS. The collection of PIFS $S=\{s_i: X\rightarrow X\}_{i\in I}$ satisfying the following conditions is called $\Gamma_X(\theta, V, \gamma, u, M)$ in \cite{SSU1}. 

\begin{enumerate}

\item There exists an open connected neighbourhood $V$ of $v$, such that
\begin{center}
$\cup_{i\in I\setminus\{i_1\}} s_i(X)\cap V=\emptyset$. 
\end{center}  

\item There exists some $\gamma\in (0,1)$, such that 
\begin{center}
$\sup_{i\in I\setminus\{i_1\}}\{\Vert s_i'\Vert\}\leq \gamma$. 
\end{center}

\item There exists some $u\in (0,1)$, such that 
\begin{center}
$\inf_{i\in I, x\in X}\{|s_i'(x)|\}\geq u$. 
\end{center}

\item There exists $M>0$ such that
\begin{center}
$\Vert S'\Vert_\theta=\sup_{i\in I}\sup_{x,y\in X}\Big\{\cfrac{|s_i'(x)-s_i'(y)|}{|x-y|^\theta}\Big\}\leq M$.
\end{center}

\end{enumerate}

For a finite PIFS, $S\in\Gamma_X(\theta)$ is enough to guarantee $S\in\Gamma_X(\theta, V, \gamma, u, M)$ for some parameters  $V, \gamma, u, M$. However, this is not true for an infinite PIFS. Our technique of approximating measures on the infinite symbolic space by the sequence of concentrating measures on its finite subspaces successfully overcomes these obstacles. 

Our next main result deals with local dimension of the exceptional parameters of a family of infinite PIFS $\{S^\bold{t}\}_{\bold{t}\in U}$. Let $\mu$ be a measure on the symbolic space $\mathbb{N}^\infty$ with its concentrating measures $\{\mu_n\}_{n\in\mathbb{N}}$. For a subset $G\subset U$ and $0<\alpha<1$, in case the sequence $\Big\{\cfrac{h_{\mu_n}(\sigma)}{\lambda^{\bold{t}}_{\mu_n}(\sigma)}\Big\}_{n=1}^\infty$ admits a limit for any $\bold{t}\in G$,  let
\begin{center}
$E_{\alpha, G}:=\Big\{\bold{t}\in G: dim^* \nu_\bold{t}<\lim_{n\rightarrow\infty}\min\big\{ \cfrac{h_{\mu_n}(\sigma)}{\lambda^{\bold{t}}_{\mu_n}(\sigma)},\alpha\big\}\Big\}$ 
\end{center}
be the level-$\alpha$ exceptional parameters in $G$. Let 
\begin{center}
$K_{\alpha, G}=\min\Big\{\sup_{\bold{t}\in G} \lim_{n\rightarrow\infty}\cfrac{h_{\mu_n}(\sigma)}{\lambda^{\bold{t}}_{\mu_n}(\sigma)}, \alpha\Big\}+d-1$.
\end{center}
For a parameter set $A\subset U\subset \mathbb{R}^d$, let $N_r(A)$ be the least number of open balls of radius $r>0$ needed to cover the set $A$. A family of PIFS $S^{\bold{t}}=\{s_i^{\bold{t}}: X\rightarrow X\}_{i\in \mathbb{N}}$ is said to satisfy the \emph{strong transversality condition}  if there exists a constant $C_2>0$ such that
\begin{center}
$N_r\{\bold{t}\in U: |\pi_\bold{t}(\omega)-\pi_\bold{t}(\tau)|\leq r\}\leq C_2r^{1-d}$
\end{center}      
for any $\omega, \tau\in \mathbb{N}^\infty, \omega_1\neq \tau_1$ and any $r>0$. Obviously this condition is stronger than the transversality condition. We have the following estimation on the upper bound of the Hausdorff dimension of the local level-$\alpha$ exceptional set. 
 
\begin{theorem}\label{thm12}
Let 
\begin{center}
$\big\{S^\bold{t}=\{s_i^{\bold{t}}: X\rightarrow X\}_{i\in \mathbb{N}}\in \Gamma_X(\theta)\big\}_{\bold{t}\in U}$
\end{center}
be a family of infinite parabolic iterated function systems satisfying the continuity and strong transversality condition with respect to the vector-time parameter $\bold{t}\in U$. For an ergodic probability measure $\mu$ on the symbolic space $\mathbb{N}^\infty$ with $\mu(\cup_{n=1}^\infty \mathbb{N}_n^\infty)=1$, if the sequence of infinite random variables in $\bold{Y}$ under the law $\mu$ is independent, then 
\begin{equation}\label{eq38}
HD(E_{\alpha, G})\leq K_{\alpha, G}
\end{equation}
for any $0<\alpha<1$ and $G\subset U$ such that $\lim_{n\rightarrow\infty} \cfrac{h_{\mu_n}(\sigma)}{\lambda^{\bold{t}}_{\mu_n}(\sigma)}$ exists for any $\bold{t}\in G$, in which $\{\mu_n\}_{n\in\mathbb{N}}$ are the concentrating measures.
\end{theorem}

This result is a generalization of \cite[Theorem 5.3]{SSU1} from families of finite PIFS to families of infinite PIFS. See also \cite[Theorem]{Kau}.

\section{The setwise topology and semi-continuity of the measure-dimension mappings under it}

Since we are considering families of IFS in this work, it would be very helpful if the  measure-dimension mappings 
\begin{center}
$dim_*: \mathcal{\hat{M}}(X)\rightarrow[0,\infty)$ 
\end{center}
or
\begin{center}
 $dim^*: \mathcal{\hat{M}}(X)\rightarrow[0,\infty)$ 
\end{center}
admits some continuity property. To discuss the continuity problem, some appropriate topology on $\mathcal{\hat{M}}(X)$ is needed. Since the measures $\{\nu_\bold{t}\}_{\bold{t}\in U}\subset \mathcal{\hat{M}}(X)$ in Theorem  \ref{thm17} inherit the parametrization naturally from parametrization of the PIFS, one may try to endow the Euclidean metric on $\{\nu_\bold{t}\}_{\bold{t}\in U}$ through the parameter $\bold{t}\in U\subset\mathbb{R}^d$ to discuss the continuity of the measure-dimension mappings. The measure-dimension mappings are not always continuous under the Euclidean metric on $\{\nu_\bold{t}\}_{\bold{t}\in U}$ in general consideration, although lower semi-continuity of them can be expected. One can see from the family of Bernoulli convolutions $\{\nu_\lambda\}_\lambda\subset \mathcal{\hat{M}}([0,1])$ with the contraction ratio $\lambda\in(0,1)$. The measure-dimension mappings   
\begin{center}
$dim_*=dim^*: (0,1)\rightarrow[0,1]$ 
\end{center}
are not always continuous since there are dimension drops from $1$ at the inverses of the Pisot parameters. However, they are lower semi-continuous \cite[Theorem 1.8]{HS}. 

Since we are dealing with measures originated from families of the infinite systems and their finite sub-systems simultaneously in this work, we will abandon the  Euclidean metric on $\{\nu_\bold{t}\}_{\bold{t}\in U}\subset \mathcal{\hat{M}}(X)$ to discuss the continuity problem. A well-known topology on $\mathcal{\hat{M}}(X)$ is the weak topology.

\begin{definition}
A sequence of bounded measures $\{\nu_n\in\mathcal{\hat{M}}(X)\}_{n=1}^\infty$ is said to converge \emph{weakly} to $\nu\in\mathcal{\hat{M}}(X)$, if 
\begin{center}
$\lim_{n\rightarrow\infty }\int_X f(x) d\nu_n=\int_X f(x) d\nu$
\end{center}
for any bounded continuous $f: X\rightarrow\mathbb{R}$. 
\end{definition}

Denote the convergence in this sense by $\nu_n\stackrel{w}{\rightarrow}\nu$ as $n\rightarrow\infty$.  However, The measure-dimension mappings $dim^*$ and $dim_*$ do not admit any semi-continuity under the weak topology on $\mathcal{\hat{M}}(X)$ \cite[Theorem 3.1]{Ma}. It turns out that the setwise topology  is the most ideal topology for us to discuss the continuity of the measure-dimension mappings.

\begin{definition}
A sequence of measures $\{\nu_n\in\mathcal{\hat{M}}(X)\}_{n=1}^\infty$ is said to converge \emph{setwisely} to $\nu\in\mathcal{\hat{M}}(X)$, if 
\begin{center}
$\lim_{n\rightarrow\infty }\nu_n(B)=\nu(B)$
\end{center}
for any $B\in\mathcal{B}$. 
\end{definition}

One is recommended to refer to \cite{Doo, FKZ, GR, HL, Las, LY} for more equivalent descriptions of the setwise topology on $\mathcal{\hat{M}}(X)$ and its various applications. Denote the convergence in this sense by $\nu_n\stackrel{s}{\rightarrow}\nu$ as $n\rightarrow\infty$. The measure-dimension mappings        
$dim^*$ and $dim_*$ are lower semi-continuous and upper semi-continuous  under the setwise topology respectively on $\mathcal{\hat{M}}(X)$.

\begin{theorem}\label{thm13}
The measure-dimension mapping $dim^*$ is lower semi-continuous, while $dim_*$ is upper semi-continuous  under the setwise topology on $\mathcal{\hat{M}}(X)$, that is, if $\nu_n\stackrel{s}{\rightarrow}\nu$ in $\mathcal{\hat{M}}(X)$ as $n\rightarrow\infty$, then 
\begin{equation}
\liminf_{n\rightarrow\infty} dim^* \nu_n\geq dim^* \nu
\end{equation}
while
\begin{equation}
\limsup_{n\rightarrow\infty} dim_* \nu_n\leq dim_* \nu.
\end{equation}
\end{theorem}

\begin{rem}
Similar results appear in \cite[Theorem 2.8, 2.9]{Ma} when one considers the measure-dimension mappings        
$dim^*$ and $dim_*$ on the probability space 
\begin{center}
$\mathcal{M}(X)=\{\nu: \nu(X)=1, \nu\in\mathcal{\hat{M}}(X)\}$
\end{center}
with $X$ being an arbitrary metric space. The proofs apply to the measure-dimension mappings on the space of finite measures $\mathcal{\hat{M}}(X)$ here, or even on the space of infinite measures. 
\end{rem}

Considering generalizations of measures,  Theorem \ref{thm13} is still true on the space of  \emph{finitely additive measures} (a finitely additive measure $\nu$ on $(X, \mathcal{B})$ is a function from $\mathcal{B}$ to $[0,\infty)$ satisfying finite additivity instead of $\sigma$-additivity), but not true for a \emph{signed measure} on $(X, \mathcal{B})$ under the Definition \ref{def3}. Obviously a new meterage for the dimension of signed measures is needed. Let $\mathcal{\hat{M}}_s(X)$ be the collection of all the finite signed Borel measures on $X$.
\begin{definition}
For a finite signed measure $\nu\in \mathcal{\hat{M}}_s(X)$,  the \emph{lower} and \emph{upper Hausdorff dimension} of $\nu$ are defined respectively to be:

\begin{center}
$dim_{s} \nu=\inf\{HD(A): \nu(A)>0 \mbox{ or } \nu(A)<0, A\in\mathcal{B}\},$  
\end{center}
and
\begin{center}
$dim^{s} \nu=\max\Big\{\inf\big\{HD(A): \nu(A)=\max\{\nu(B)\}_{B\in\mathcal{B}}\big\}, \inf\big\{HD(A): \nu(A)=\min\{\nu(B)\}_{B\in\mathcal{B}}\big\}\Big\}.$
\end{center}
\end{definition} 

Under this modified definition the measure-dimension mappings $dim_{s}$ and  $dim^{s}$  from $\mathcal{\hat{M}}_s(X)$ to $[0,\infty)$ still bear some continuity.

\begin{proposition}
The measure-dimension mapping $dim^s$ is lower semi-continuous, while $dim_s$ is upper semi-continuous  under the setwise topology on $\mathcal{\hat{M}}_s(X)$, that is, if $\nu_n\stackrel{s}{\rightarrow}\nu$ in $\mathcal{\hat{M}}_s(X)$ as $n\rightarrow\infty$, then 
\begin{equation}\label{eq45}
\liminf_{n\rightarrow\infty} dim^s \nu_n\geq dim^s \nu
\end{equation}
while
\begin{equation}
\limsup_{n\rightarrow\infty} dim_s \nu_n\leq dim_s \nu.
\end{equation}
\end{proposition}
\begin{proof}
The conclusion $dim_s$ is upper semi-continuous follows a similar argument as the proof of \cite[Theorem 2.9]{Ma}. Now we show $dim^s$ is lower semi-continuous under the setwise topology on $\mathcal{\hat{M}}_s(X)$.   Without loss of generality we assume 
\begin{center}
$\inf\big\{HD(A): \nu(A)=\max\{\nu(B)\}_{B\in\mathcal{B}}\big\}\geq\inf\big\{HD(A): \nu(A)=\min\{\nu(B)\}_{B\in\mathcal{B}}\big\}$
\end{center}
and
\begin{center}
$\inf\big\{HD(A): \nu(A)=\max\{\nu(B)\}_{B\in\mathcal{B}}\big\}>0$, $\max\{\nu(B)\}_{B\in\mathcal{B}}>0$.
\end{center}
Suppose conversely that (\ref{eq45}) does not hold. Then we can find a subsequence $\{n_i\}_{i=1}^\infty$ and a real $0<a<\inf\big\{HD(A): \nu(A)=\max\{\nu(B)\}_{B\in\mathcal{B}}\big\}$ such that
\begin{equation}\label{eq44}
\lim_{i\rightarrow\infty} dim^s \nu_{n_i}<a< dim^s \nu=\inf\big\{HD(A): \nu(A)=\max\{\nu(B)\}_{B\in\mathcal{B}}\big\}.
\end{equation}
Now apply the Hahn decomposition theorem to the measure $\nu$, let $X^+\subset X$ be a positive set of $\nu$ with $\nu(X^+)=\max\{\nu(B)\}_{B\in\mathcal{B}}$. Under the above assumption we have 
\begin{center}
$HD(X^+)\geq dim^s \nu>a>dim^s \nu_{n_i}$ 
\end{center}
for any $i$ large enough. Since $\lim_{i\rightarrow\infty} \nu_{n_i}(X^+)=\nu(X^+)>0$, considering (\ref{eq44}), for any $i$ large enough, there must exist $X^+_i\subset X^+$, such that
\begin{center}
$\nu_{n_i}(X^+_i)=0$ and $HD(X^+_i)=HD(X^+)>a>HD(X^+\setminus X^+_i)$.
\end{center}
Now let $X^+_\infty=\cup_{i=N}^\infty (X^+\setminus X^+_i)$ for some integer $N$ large enough. One can see that 
\begin{center}
$X^+\setminus X^+_\infty\neq\emptyset$ 
\end{center}
as 
\begin{equation}\label{eq43}
HD(X^+_\infty)\leq a<HD(X^+).
\end{equation}
Moreover, we have
\begin{center}
$\nu(X^+\setminus X^+_\infty)=\lim_{i\rightarrow\infty} \nu_{n_i}(X^+\setminus X^+_\infty)=0$,
\end{center}
which forces
\begin{equation}\label{eq42}
\nu(X^+_\infty)=\nu(X^+).
\end{equation}
Now (\ref{eq43}) together with (\ref{eq42}) contradict the fact that $dim^s \nu>a$, which finishes the proof.

\end{proof}

\section{Absolute continuity of convergent sequences of measures under the setwise topology}\label{sec2}

In  this section we discuss the relationship between absolute continuity of a sequence of measures and absolute continuity of its limit measure (under the weak or setwise topology, in case the sequence converges). We start from the following basic result.

\begin{proposition}\label{thm14}
For a sequence of measures $\nu_n\stackrel{s}{\rightarrow}\nu$ in $\mathcal{\hat{M}}(X)$ as $n\rightarrow\infty$, if there exists a subsequence  $\{\nu_{n_i}\}_{i=1}^\infty$ such that $\nu_{n_i}$ is absolutely continuous with respect to some $\varrho\in \mathcal{\hat{M}}(X)$ for any $1\leq i<\infty$, then $\nu$ is absolutely continuous with respect to $\varrho$.
\end{proposition}
\begin{proof}
We do this by reduction to absurdity. Suppose that $\nu$ is not absolutely continuous with respect to $\varrho$, then there exists $A\in\mathcal{B}$ such that $\nu(A)>0$ while $\varrho(A)=0$. Since $\nu_n\stackrel{s}{\rightarrow}\nu$ as $n\rightarrow\infty$, there exists $N$ large enough such that $\nu_{n_N}(A)>0$. Considering $\varrho(A)=0$, this contradicts the fact that $\nu_{n_N}$ is absolutely continuous with respect to $\varrho$. 

\end{proof}

Due to Proposition \ref{thm14}, we make the following definition formally.
\begin{definition}
A sequence of  measures $\{\nu_n\}_{n=1}^\infty$ is said to be \emph{absolutely continuous} with respect to some $\varrho$ on $\mathcal{\hat{M}}(X)$  if it contains a subsequence such that every measure in the subsequence is absolutely continuous with respect to $\varrho$.
\end{definition}

In fact Proposition \ref{thm14} holds for any topological ambient space $X$. However, it is possible that a sequence of absolutely continuous measures converges weakly to a measure which is not absolutely continuous, as one can see from the Example \ref{exm6}. Let  $\mathfrak{L}^d|_{A}$ be the restriction of $\mathfrak{L}^d$ on a set $A\subset\mathbb{R}^d$.

\begin{exm}\label{exm6}
Consider the Cantor middle-third set $J$ appears as the attractor of the homogeneous IFS 
\begin{center}
$S=\big\{s_1(x)=\frac{1}{3}x,s_2(x)=\frac{1}{3}x+\frac{2}{3}\big\}$
\end{center}
on $X=[0,1]$. Let 
\begin{center}
$\nu_n=\frac{1}{2^n} \mathfrak{L}^1|_{\cup_{\omega\in\mathbb{N}_2^n}s_\omega(X)}$.
\end{center}
Let  $\nu$ be the unique probability such that 
\begin{center}
$\nu=\frac{1}{2}\nu\circ s_1^{-1}+\frac{1}{2}\nu\circ s_2^{-1}$. 
\end{center}

\end{exm}

Unfortunately, the inverse of Proposition \ref{thm14} is not always true, as one can see from the following example. Let $\delta_{\{x\}}$ be the Dirac probability measure at the point $x\in X$.

\begin{exm}\label{exm7}
For $X=[0,1]$, let 
\begin{center}
$\nu_n=\mathfrak{L}^1|_{X}+\frac{1}{n}\delta_{\{1\}}$
\end{center}
be a sequence of measures in $\mathcal{\hat{M}}(X)$. 
\end{exm}

It is easy to justify that $\nu_n\stackrel{s}{\rightarrow}\mathfrak{L}^1|_{X}$  as $n\rightarrow\infty$ in Example \ref{exm7}. However, there are no  measure absolutely continuous with respect to $\mathfrak{L}^1|_{X}$ at all in the sequence $\{\nu_n\}_{n=1}^\infty$. One can see that in Example \ref{exm7} the measures in the sequence are all of mixed type. In fact this is the only reason which prevents the inverse of Proposition \ref{thm14} to be true.

\begin{proposition}\label{pro1}
For a sequence of measures $\nu_n\stackrel{s}{\rightarrow}\nu$ in $\mathcal{\hat{M}}(X)$ as $n\rightarrow\infty$ with $\nu(X)>0$, if $\nu$ is absolutely continuous with respect to some $\varrho\in\mathcal{\hat{M}}(X)$ and the measures in the sequence are all of pure type with respect to $\varrho$, then there exists $N\in\mathbb{N}$ large enough such such that $\nu_n$ is absolutely continuous with respect to $\varrho\in \mathcal{\hat{M}}(X)$ for any $n\geq N$.
\end{proposition}
\begin{proof}
Suppose in the contrary that the conclusion is not true, then we can find a subsequence $\{n_i\}_{i=1}^\infty\subset \mathbb{N}$ such that any measure in the  sequence $\{\nu_{n_i}\}_{i=1}^\infty$ is singular with respect to $\varrho$. Denote an essential support of $\nu_{n_i}$ by $A_i\subset X$ with $\varrho(A_i)=0$ for any $1\leq i<\infty$. Let $A=\cup_{i=1}^\infty A_i$, according to the $\sigma$-additivity of $\varrho$, we have 
\begin{equation}\label{eq47}
\varrho(A)=0.
\end{equation}
Note that since $\nu_n\stackrel{s}{\rightarrow}\nu$ as $n\rightarrow\infty$, we have
\begin{equation}\label{eq46}
\nu(A)=\lim_{i\rightarrow\infty} \nu_{n_i}(A)=\lim_{i\rightarrow\infty} \nu_{n_i}(X)=\nu(X)>0.
\end{equation}
Now (\ref{eq47}) together with (\ref{eq46}) contradict the fact that $\nu$ is absolutely continuous with respect to $\varrho$.  
\end{proof}

Proposition \ref{pro1} also holds for any topological ambient space $X$, but it is not true for convergent sequences of measures under the weak topology on $\mathcal{\hat{M}}(X)$, as one can also see from our Example \ref{exm6}. The above results show that the absolute continuity of convergent sequences of measures and their limit measures are equivalent to each other under the setwise topology on $\mathcal{\hat{M}}(X)$, while this relationship is not true under the weak topology. This is another advantage that the setwise topology takes over the weak topology on $\mathcal{\hat{M}}(X)$.

\section{The $n$-th concentrating measures of probability on the symbolic spaces and some inherited properties}\label{sec4}

In this section we do the approximation of a probability measure $\mu$ on the infinite symbolic space $\mathbb{N}^\infty$ by the sequence of its concentrating measures $\{\mu_n\}_{n=1}^\infty$ supported on $\{\mathbb{N}_n^\infty\}_{n=1}^\infty$ respectively. We will prove that some properties are inherited by the concentrating measures from $\mu$, under the hypothesis that the sequence of infinite random variables in the infinite stochastic process 
\begin{center}
$\bold{Y}=\{Y_i\}_{i\in\mathbb{N}}$
\end{center}
is independent  under the law $\mu$. As indicated in Remark \ref{rem2}, one may be able to remove the independent hypothesis in due course if one tries the approximation of $\mu$ by its restrictions $\{\mu|_{\mathbb{N}_n^\infty}\}_{n=1}^\infty\subset\mathcal{\hat{M}}(\mathbb{N}_n^\infty)$ instead of the concentrating measures $\{\mu_n\}_{n=1}^\infty\subset\mathcal{M}(\mathbb{N}_n^\infty)$.  

\begin{definition}\label{def4}
For a probability measure $\mu$ on the symbolic space $\mathbb{N}^\infty$, define its $n$-th \emph{concentrating measure} $\mu_n$ for any $n\in\mathbb{N}$ to be the unique probability measure supported on $\mathbb{N}_n^\infty$ satisfying the following conditions.

\begin{enumerate}[(1)]
\item   On the first level cylinders,
\begin{center}
$\mu_n([i])=\mu([i])$ for $1\leq i\leq n-1$,  $\mu_n([n])=\sum_{i=n}^\infty\mu([i])$.
\end{center}

\item On the second level cylinders,
\begin{center}
$
\begin{array}{l}
\mu_n([ij])=\mu([ij]) \mbox{ for } 1\leq i,j\leq n-1,\\
\mu_n([in])=\sum_{k=n}^\infty\mu([ik]) \mbox{ for } 1\leq i\leq n-1, \\
\mu_n([nj])=\sum_{k=n}^\infty\mu([kj]) \mbox{ for } 1\leq j\leq n-1,\\
\mu_n([nn])=\sum_{l=n}^\infty\sum_{k=n}^\infty\mu([kl]).
\end{array}
$ 
\end{center}

\item For any $q$-th cylinder $[i_1i_2\cdots i_q]\subset \mathbb{N}_n^\infty$ with $i_1,i_2,\cdots, i_q\in \mathbb{N}_n$ and $q\in \mathbb{N}$, let 
\begin{center}
$\#\{1\leq k\leq q: i_k=n\}=l$. 
\end{center}
Number the indexes in $\{1\leq k\leq q: i_k=n\}$ in increasing order as
\begin{center}
$j_1<j_2<\cdots<j_l$,
\end{center} 
that is, $\{j_1,j_2,\cdots,j_l\}=\{1\leq k\leq q: i_k=n\}$. Then
\begin{center}
$
\begin{array}{ll}
& \mu_n([i_1i_2\cdots i_q]) \\
= & \sum_{n\leq r_1, r_2, \cdots, r_l<\infty} \mu([i_1\cdots i_{j_1-1}r_1 i_{j_1+1}\cdots i_{j_2-1}r_2 i_{j_2+1}\cdots i_{j_l-1}r_l i_{j_l+1}\cdots  i_q]).
\end{array}
$
\end{center} 
\end{enumerate}

\end{definition}

Since for any $q$-level cylinder $[i_1i_2\cdots i_q]\subset \mathbb{N}_n^\infty$ with $i_1,i_2,\cdots i_q\in \mathbb{N}_n$ and $q\in \mathbb{N}$, the conditions in Definition \ref{def4} (1)-(3) imply
\begin{center}
$\mu_n([i_1i_2\cdots i_q])=\sum_{j=1}^n\mu_n([i_1i_2\cdots i_q j])$
\end{center}
between  cylinders of successive levels, the existence and uniqueness of the measure $\mu_n$  are guaranteed by the \emph{Kolmogorov extension theorem}, see for example \cite[Theorem 2.4.3]{Tao}. Moreover, $\mu_n$ is invariant for any $n\in \mathbb{N}$ if $\mu$ is invariant with respect to the shift map $\sigma$ according to \cite[Theorem 2]{Ber}. In fact, more properties are inherited from $\mu$ as the structure of the original measure  is suitably preserved considering the structure of the  concentrating measures $\{\mu_n\}_{n\in\mathbb{N}}$  on the cylinder sets, under the independent hypothesis of $\bold{Y}$ with respect to $\mu$.

\begin{lemma}\label{lem15}
For a measure $\mu$ on the symbolic space $\mathbb{N}^\infty$, if the sequence of infinite random variables in $\bold{Y}$ under the law $\mu$ is independent, then it is also independent under the concentrating law $\mu_n$ for any $n\in \mathbb{N}$.
\end{lemma}
\begin{proof}
For any $q$-word $\omega=i_1i_2\cdots i_q\in \mathbb{N}_n^q$, let $\#\{1\leq k\leq q: i_k=n\}=l\geq 0$ and
\begin{center}
$\{j_1,j_2,\cdots,j_l\}=\{1\leq k\leq q: i_k=n\}$.
\end{center}

Since $\bold{Y}$ is independent under the law $\mu$, we have
\begin{center}
$
\begin{array}{ll}
& \mu_n([\omega])  \\
= & \sum_{n\leq r_1< \infty, n\leq r_2< \infty, \cdots, n\leq r_l< \infty} \mu([i_1\cdots i_{j_1-1}r_1 i_{j_1+1}\cdots i_{j_2-1}r_2 i_{j_2+1}\cdots i_{j_l-1}r_l i_{j_l+1}\cdots  i_q])\\
= & \mu([i_1])\cdots \mu([i_{j_1-1}])\big(\sum_{n\leq r_1< \infty}\mu([r_1])\big)\mu([i_{j_1+1}])\cdots\mu([i_{j_2-1}])\big(\sum_{n\leq r_2< \infty}\mu([r_2])\big)\mu([i_{j_2+1}])\\
&\cdots \mu([i_{j_l-1}])\big(\sum_{n\leq r_l< \infty}\mu([r_l])\big)\mu([i_{j_l+1}])\cdots  \mu([i_q])\\
= & \mu_n([i_1])\cdots \mu_n([i_{j_1-1}])\mu_n([i_{j_1}])\mu_n([i_{j_1+1}])\cdots\mu_n([i_{j_2-1}])\mu_n([i_{j_2}])\mu_n([i_{j_2+1}])\cdots \\
& \mu_n([i_{j_l-1}])\mu_n([i_{j_l}])\mu_n([i_{j_l+1}])\cdots  \mu_n([i_q]),
\end{array}
$
\end{center}
which justifies the independence of the random variables in $\bold{Y}$ under the law $\mu_n$ for any $n\in \mathbb{N}$.
\end{proof}

The ergodicity is also inherited by the concentrating measures $\{\mu_n\}_{n\in\mathbb{N}}$ under the independent hypothesis of $\bold{Y}$.

\begin{lemma}\label{lem16}
For an ergodic measure $\mu$ on the symbolic space $\mathbb{N}^\infty$, if the sequence of infinite random variables in $\bold{Y}$ under the law $\mu$ is independent, then  the concentrating measure $\mu_n$ is also ergodic on $\mathbb{N}_n^\infty$ for any $n\in \mathbb{N}$ with respect to the shift map $\sigma$.
\end{lemma}
\begin{proof}
This is a standard proof in classical probability techniques, with all the above preparations. First, due to Lemma  \ref{lem15} and the independence assumption, the sequence of random variables $\{Y_i\}_{i=1}^\infty$ in  $\bold{Y}$ under the law $\mu_n$ is independent for any $n\in \mathbb{N}$. Now let $\mathcal{B}_{\mathbb{N}_n^\infty}$ be the $\sigma$-algebra generated by the cylinder sets in $\mathbb{N}_n^\infty$ for some $n\in \mathbb{N}$. Let
\begin{center}
$\mathcal{B}_k=\{\mathbb{N}_n^k\times B: B\in \mathcal{B}_{\mathbb{N}_n^\infty}\}$.
\end{center}  
The tail $\sigma$-algebra is define to be 
\begin{center}
$\mathcal{B}_T=\cap_{k\in\mathbb{N}}\mathcal{B}_k$.
\end{center} 
Obviously the events
\begin{center}
$\{A\in \mathcal{B}_{\mathbb{N}_n^\infty}: \sigma^{-1}(A)=A\}\subset \mathcal{B}_T$
\end{center}
are all tail events. So according to the Kolmogorov’s $0-1$ law (see for example \cite[Proposition 3b9]{Tsi} or \cite{Shi}),
\begin{center}
$\mu_n(A)=0$ or $1$
\end{center}
for any event in $\{A\in \mathcal{B}_{\mathbb{N}_n^\infty}: \sigma^{-1}(A)=A\}$.
\end{proof}

\begin{rem}
Similar to Remark  \ref{rem3}, we are wondering whether the result holds without the hypothesis of independence on the sequence of infinite random variables with law $\mu$, which essentially confines our main results to be under this condition. 
\end{rem}
Considering the limit behaviour of the sequence of concentrating measures $\{\mu_n\}_{n\in\mathbb{N}}$, we have the following simple but important result.

\begin{proposition}\label{pro2}
For a probability measure $\mu$ on the symbolic space $\mathbb{N}^\infty$ such that $\mu(\cup_{n=1}^\infty \mathbb{N}_n^\infty)=1$, let $\mu_n$ be its $n$-th concentrating measure  on $\mathbb{N}_n^\infty$ for any $n\in \mathbb{N}$. Then
\begin{center}
$\mu_n\stackrel{s}\rightarrow \mu$ 
\end{center}
as $n\rightarrow\infty$.
\end{proposition} 
\begin{proof}
It is obvious that for any finite word $\omega\in \mathbb{N}^*$, we have
\begin{center}
$\lim_{n\rightarrow\infty}\mu_n([\omega])=\mu([\omega])$,
\end{center}
as $\mu_n([\omega])=\mu([\omega])$ for any $n$ large enough. Since $\mu$ is supported on $\cup_{n=1}^\infty \mathbb{N}_n^\infty$, the convergence extends to all measurable sets in $\mathbb{N}^\infty$ due to regularity of the measures, or \cite[Theorem 2.3]{FKZ}.
\end{proof}

\begin{rem}\label{rem4}
The convergence of the sequence of concentrating measures $\{\mu_n\}_{n\in\mathbb{N}}$ of a measure $\mu$ is usually not true under the total variation (TV) topology on $\mathcal{M}(\mathbb{N}^\infty)$, except in some trivial cases. 
\end{rem}

The entropy is also inherited by the sequence of concentrating measures $\{\mu_n\}_{n\in\mathbb{N}}$ from the original measure $\mu$, considering Proposition \ref{pro2} as well as the following result.

\begin{lemma}\label{lem12}
For a finite measure $\mu\in\mathcal{\hat{M}}(\mathbb{N}^\infty)$ and its concentrating measures $\{\mu_n\}_{n\in\mathbb{N}}$,  we have
\begin{center}
$\lim_{n\rightarrow\infty}h_{\mu_n}(\sigma)=h_\mu(\sigma)$
\end{center}
with respect to $(\mathbb{N}^\infty, \mathcal{B}_{\mathbb{N}^\infty})$.
\end{lemma}

\begin{proof}
We justify the result in case of $h_\mu(\sigma)=\infty$ and $h_\mu(\sigma)<\infty$ respectively.

If 
\begin{center}
$h_\mu(\sigma)=\liminf_{k\rightarrow\infty} -\frac{1}{k}\sum_{\omega\in\mathbb{N}^k} \mu([\omega])\log\mu([\omega])=\infty$, 
\end{center}
we have 
\begin{center}
$-\frac{1}{k}\sum_{\omega\in\mathbb{N}^k} \mu([\omega])\log\mu([\omega])=\infty$
\end{center} 
for any $k\in\mathbb{N}$. Then for any $a>0$ and $k\in\mathbb{N}$, we can find $N_{a,k}$ large enough, such that 
\begin{center}
$-\frac{1}{k}\sum_{\omega\in\mathbb{N}_n^k} \mu_n([\omega])\log\mu_n([\omega])>a$
\end{center} 
for any $n>N_{a,k}$. This is enough to force $\lim_{n\rightarrow\infty}h_{\mu_n}(\sigma)=h_\mu(\sigma)=\infty$ by reduction to absurdity. 

Now if 
\begin{center}
$h_\mu(\sigma)=\liminf_{k\rightarrow\infty} -\frac{1}{k}\sum_{\omega\in\mathbb{N}^k} \mu([\omega])\log\mu([\omega])<\infty$, 
\end{center}
then for any small $\epsilon>0$, we can find $K_\epsilon$ large enough, such that
\begin{center}
$h_\mu(\sigma)\leq -\frac{1}{k}\sum_{\omega\in\mathbb{N}^k} \mu([\omega])\log\mu([\omega])<h_\mu(\sigma)+\epsilon$
\end{center}
for any $k> K_\epsilon$. Now for fixed $k> K_\epsilon$, since $-\sum_{\omega\in\mathbb{N}^k} \mu([\omega])\log\mu([\omega])<\infty$, we can find $N_{k,\epsilon}$ large enough such that 
\begin{center}
$\big| \sum_{\omega\in\mathbb{N}^k} \mu([\omega])\log\mu([\omega])-\sum_{\omega\in\mathbb{N}_n^k} \mu_([\omega])\log\mu_n([\omega])\big| <\epsilon$
\end{center}
for any $n>N_{k,\epsilon}$. This is enough to force $\lim_{n\rightarrow\infty}h_{\mu_n}(\sigma)=h_\mu(\sigma)$.
  
\end{proof}

\begin{rem}\label{rem1}
In fact Lemma \ref{lem12} holds for any sequence of measures converging to the finite measure $\mu$ under the setwise topology in $\mathcal{\hat{M}}(X)$, not only for the sequence of its concentrating measures, for any topological ambient space $X$. The version now is enough for our purpose in this work. 
\end{rem}

\section{Setwise approximation of projective measures on the attractors of families of infinite PIFS}\label{sec9}

In this section we aim at proving Theorem \ref{thm17}. After the proof we formulate some interesting applications of Theorem \ref{thm17} in some circumstances. We first show a result which links the upper and lower dimensions of the projective measure $\nu$ on the attractor $J$ of a single PIFS.

\begin{lemma}\label{lem17}
Let 
\begin{center}
$S=\{s_i: X\rightarrow X\}_{i\in I}$
\end{center}
be a parabolic iterated function system with a countable index set $I$. For an ergodic measure $\mu$ on $I^\infty$, let $\nu=\mu\circ\pi^{-1}$ be its projective measure. Then
\begin{center}
$dim_* \nu=dim^* \nu$.
\end{center}
\end{lemma}
\begin{proof}
To prove the result, it suffices for us to prove 
\begin{equation}\label{eq50}
dim_* \nu\geq dim^* \nu.
\end{equation}
Let $J$ be the attractor of the PIFS. For a set $A\subset J$ with $\nu(A)=\mu\circ\pi^{-1}(A)>0$, consider the set 
\begin{center}
$\pi\circ\sigma^{-1}\circ\pi^{-1}(A)=\pi(\cup_{i\in I}i\pi^{-1}(A))=\cup_{i\in I}\pi(i\pi^{-1}(A))$.
\end{center} 
Note that for any $i\in I$, due to $\pi(\omega)=s_{\omega|_n}\circ\pi\circ\sigma^n(\omega)$ for any $\omega\in I^\infty$, we have
\begin{center}
$\pi(i\pi^{-1}(A))=s_i(A)$,
\end{center}
in which $i\pi^{-1}(A):=\{\omega\in I^\infty: \omega_1=i, \sigma(\omega)\in \pi^{-1}(A)\}$. Since $0< s_i'(x)\leq 1$ for any $i\in I$, we have
\begin{center}
$HD\big(\pi(i\pi^{-1}(A))\big)=HD(A)$
\end{center}
for any $i\in I$. Since $I$ is countable, this forces
\begin{center}
$HD(\pi\circ\sigma^{-1}\circ\pi^{-1}(A))=HD(A)$.
\end{center}
Successively we can show $HD(\pi\circ\sigma^{-n}\circ\pi^{-1}(A))=HD(A)$ for any $n\in\mathbb{N}$, which forces
\begin{center}
$HD(\cup_{n\in\mathbb{N}}\pi\circ\sigma^{-n}\circ\pi^{-1}(A))=HD(A)$.
\end{center}
Since $\mu$ is ergodic with respect to $\sigma$, then

\begin{center}
$
\begin{array}{ll}
& \nu\big(\cup_{n\in\mathbb{N}}\pi\circ\sigma^{-n}\circ\pi^{-1}(A)\big)\\
=&\mu\circ\pi^{-1}\big(\cup_{n\in\mathbb{N}}\pi\circ\sigma^{-n}\circ\pi^{-1}(A)\big)\\
\geq & \mu\big(\cup_{n\in\mathbb{N}}\sigma^{-n}\circ\pi^{-1}(A)\big)\\
=& 1.
\end{array}
$
\end{center}
This implies the inequality (\ref{eq50}).

\end{proof}

Lemma \ref{lem17} may not be true for some invariant measure $\mu$ on the symbolic space. Apply it to families of PIFS, we have the following result instantly.
\begin{corollary}\label{cor5}
Let 
\begin{center}
$\big\{S^{\bold{t}}=\{s_i^{\bold{t}}: X\rightarrow X\}_{i\in I}\big\}_{\bold{t}\in U}$
\end{center}
be a family of parabolic iterated function systems with a countable index set $I$. Let $\mu$ be an ergodic measure on $I^\infty$ and $\nu_\bold{t}$ be the projective measure under $\pi_\bold{t}$ at time $\bold{t}$, then
\begin{center}
$dim_* \nu_\bold{t}=dim^* \nu_\bold{t}$
\end{center}
for any $\bold{t}\in U$.
\end{corollary}

Note that the equality holds everywhere instead of Lebesgue \emph{a.e.} with respect to $\bold{t}\in U$. Now we consider setwisely approximating the projective measure of $\mu$ by the sequence of projective measures of its concentrating measures on the ambient space $X$.  The following result is an instant corollary of Proposition \ref{pro2}. 

\begin{corollary}\label{cor6}
Let 
\begin{center}
$S=\{s_i: X\rightarrow X\}_{i\in \mathbb{N}}$
\end{center}
be an infinite parabolic iterated function system with its attractor $J$. For a probability measure $\mu$ on the symbolic space $\mathbb{N}^\infty$ such that $\mu(\cup_{n=1}^\infty \mathbb{N}_n^\infty)=1$, let $\mu_n$ be its $n$-th concentrating measure  on $\mathbb{N}_n^\infty$ for any $n\in \mathbb{N}$. Then their projections $\{\nu_n=\mu_n\circ\pi^{-1}\}_{n\in\mathbb{N}}$ and $\nu=\mu\circ\pi^{-1}$ satisfy
\begin{center}
$\nu_n\stackrel{s}\rightarrow \nu$ 
\end{center}
as $n\rightarrow\infty$ on $J$. 
\end{corollary}

As to absolute continuity of the projective measures on the ambient space, we have the following result. 
\begin{lemma}\label{lem18}
Let 
\begin{center}
$S=\{s_i: X\rightarrow X\}_{i\in \mathbb{N}}$
\end{center}
be an infinite parabolic iterated function system with attractor $J$. For an ergodic probability measure $\mu$ on the symbolic space $\mathbb{N}^\infty$, with $\mu(\cup_{n=1}^\infty \mathbb{N}_n^\infty)=1$ and its $n$-th concentrating measure $\mu_n$ also being ergodic on $\mathbb{N}_n^\infty$ for any $n\in \mathbb{N}$, consider their projections $\{\nu_n=\mu_n\circ\pi^{-1}\}_{n\in\mathbb{N}}$ and $\nu=\mu\circ\pi^{-1}$ on $J$.  The sequence of measures $\{\nu_n\}_{n\in\mathbb{N}}$ is absolutely continuous with respect to $\mathfrak{L}^1$ if and only if $\nu$ is absolutely continuous with respect to $\mathfrak{L}^1$.
\end{lemma}

\begin{proof}
First note that all the projective measures $\{\nu_n\}_{n\in\mathbb{N}}$ and $\nu$ are of pure type since they are all projected from ergodic measures on the symbolic spaces. By Corollary \ref{cor6} we have  
\begin{center}
$\nu_n\stackrel{s}\rightarrow \nu$ 
\end{center}
as $n\rightarrow\infty$ on $J$. Then the  absolute continuity of the sequence of measures $\{\nu_n\}_{n\in\mathbb{N}}$ and the absolute continuity of $\nu$ with respect to $\mathfrak{L}^1$ is  equivalent to each other in virtue of Proposition \ref{thm14} and  Proposition \ref{pro1}.
\end{proof}

We will only use the result that absolute continuity of the sequence of measures $\{\nu_n\}_{n\in\mathbb{N}}$ implies the absolute continuity of $\nu$ with respect to $\mathfrak{L}^1$ in the following. This is true even if $\{\mu_n\}_{n\in\mathbb{N}}$ are not ergodic according to Proposition \ref{thm14}. Equipped with all the above results, now we are well prepared to prove Theorem \ref{thm17}.\\

Proof of Theorem \ref{thm17}:\\

\begin{proof}
First note that since $\big\{S^\bold{t}=\{s_i^{\bold{t}}: X\rightarrow X\}_{i\in \mathbb{N}}\in \Gamma_X(\theta)\big\}_{\bold{t}\in U}$,  there exist sequences of positive real numbers $\{0<\gamma_n, u_n<1\}_{n\in\mathbb{N}}$, $\{M_n>0\}_{n\in\mathbb{N}}$ (it is possible that these sequences satisfy $\sup\{\gamma_n\}_{n\in\mathbb{N}}=1$, $\inf\{u_n\}_{n\in\mathbb{N}}=0$ and $\sup\{M_n\}_{n\in\mathbb{N}}=\infty$) and a sequence of open neighbourhoods $V_n$ of $v$ (it is also possible that $\cap_{n=1}^\infty V_n=\{v\}$), such that its $n$-th family of truncates 
\begin{center}
$\big\{S_n^\bold{t}=\{s_i^{\bold{t}}: X\rightarrow X\}_{i\in \mathbb{N}_n}\in \Gamma_X(\theta, V_n, \gamma_n, u_n, M_n)\big\}_{\bold{t}\in U}$
\end{center}
for any fixed $n\in\mathbb{N}$. Denote their attractors by $\{J_{n,\bold{t}}\subset J_{\bold{t}}\}_{n\in\mathbb{N},\bold{t}\in U}$. Consider the projections of the concentrating measures 
\begin{center}
$\nu_{n,\bold{t}}=\mu_n\circ\pi_{\bold{t}}^{-1}$ 
\end{center}
on the limit sets for $\bold{t}\in U, n\in\mathbb{N}$. Since the sequence of infinite random variables
in $\bold{Y}$ is independent under the ergodic law $\mu$, the concentrating measures $\{\mu_n\}_{n\in\mathbb{N}}$ are also ergodic with respect to the shift map $\sigma$ in virtue of Lemma \ref{lem16}.  Moreover, according to Lemma \ref{lem12}, we have
\begin{center}
$h_{\mu_n}(\sigma)>0$
\end{center}
for $n$ large enough. Now apply the \cite[Theorem 2.3 (i)]{SSU1} and Corollary \ref{cor5} to the ergodic projections $\{\nu_{n,\bold{t}}\}_{\bold{t}\in U}$ originated from the family of finite PIFS
\begin{center}
$\big\{S_n^\bold{t}=\{s_i^{\bold{t}}: X\rightarrow X\}_{i\in \mathbb{N}_n}\in \Gamma_X(\theta, V_n, \gamma_n, u_n, M_n)\big\}_{\bold{t}\in U}$
\end{center}
satisfying the continuity and transversality condition, we have
\begin{center}
$dim_* \nu_{n,\bold{t}}=dim^* \nu_{n,\bold{t}}=\min\Big\{ \cfrac{h_{\mu_n}(\sigma)}{\lambda^{\bold{t}}_{\mu_n}(\sigma)},1\Big\}$
\end{center}
for Lebesgue \emph{a.e.} $\bold{t}\in U$ and any $n\in\mathbb{N}$ large enough. In virtue of Corollary \ref{cor6}, we have
\begin{center}
$\nu_{n,\bold{t}}\stackrel{s}\rightarrow \nu_\bold{t}$ 
\end{center}
as $n\rightarrow\infty$ on $J$ for any $\bold{t}\in U$. Now apply the semi-continuity result Theorem \ref{thm13} to the setwisely convergent sequence of projective measures $\{\nu_{n,\bold{t}}\}_{\bold{t}\in U}$, we have

\begin{equation}\label{eq54}
\begin{array}{ll}
\limsup_{n\rightarrow\infty}\min\Big\{\cfrac{h_{\mu_n}(\sigma)}{\lambda^{\bold{t}}_{\mu_n}(\sigma)},1\Big\}=\limsup_{n\rightarrow\infty}dim_*\nu_{n,\bold{t}}\leq dim_*\nu_{\bold{t}}\\
= dim^*\nu_{\bold{t}}\leq \liminf_{n\rightarrow\infty} dim^*\nu_{n,\bold{t}}=\liminf_{n\rightarrow\infty}\min\Big\{\cfrac{h_{\mu_n}(\sigma)}{\lambda^{\bold{t}}_{\mu_n}(\sigma)},1\Big\}
\end{array}
\end{equation}
for Lebesgue \emph{a.e.} $\bold{t}\in U$ and any $n\in\mathbb{N}$ large enough. Then Theorem \ref{thm17} (i) is proved by letting $n\rightarrow\infty$ in (\ref{eq54}), since the Lebesgue measure of unions of countablly many Lebesgue null sets is still null.

To prove Theorem \ref{thm17} (ii), for fixed $n\in\mathbb{N}$, apply \cite[Theorem 2.3 (ii)]{SSU1} to the ergodic projections $\{\nu_{n,\bold{t}}\}_{\bold{t}\in U}$ originated from the $n$-th truncated family of PIFS $\{S_n^\bold{t}\}_{\bold{t}\in U}$ satisfying the continuity and transversality condition, we see that $\nu_{n,\bold{t}}$ is absolutely continuous for \emph{a.e.} $\bold{t}\in\Big\{\bold{t}\in U: \cfrac{h_{\mu_n}(\sigma)}{\lambda^{\bold{t}}_{\mu_n}(\sigma)}>1\Big\}$. So in virtue of Lemma \ref{lem18},  $\nu_{\bold{t}}$ is absolutely continuous for \emph{a.e.} $\bold{t}$ in
\begin{center}
$\cup_{\{n_j\}_{j=1}^\infty \mbox{ is an infinite subsequence of } \{n\}_{n=1}^\infty} \cap_{j=1}^\infty\Big\{\bold{t}\in U: \cfrac{h_{\mu_{n_j}}(\sigma)}{\lambda^{\bold{t}}_{\mu_{n_j}}(\sigma)}>1\Big\}:=U_a$.
\end{center}
It is easy to see that  
\begin{center}
$\Big\{\bold{t}\in U: \limsup_{n\rightarrow\infty}\cfrac{h_{\mu_n}(\sigma)}{\lambda^{\bold{t}}_{\mu_n}(\sigma)}>1\Big\}\subset U_a$. 
\end{center}
This justifies Theorem \ref{thm17} (ii).
\end{proof}

One can see from the above proof that the limit of the sequence
\begin{center}
$\Big\{\min\big\{ \cfrac{h_{\mu_n}(\sigma)}{\lambda^{\bold{t}}_{\mu_n}(\sigma)},1\big\}\Big\}_{n=1}^\infty$
\end{center}
always exists for Lebesgue \emph{a.e.} $\bold{t}\in U$. More interesting  results are possible according to different \emph{a.e.} limit behaviours of the sequence $\Big\{\cfrac{h_{\mu_n}(\sigma)}{\lambda^{\bold{t}}_{\mu_n}(\sigma)}\Big\}_{\bold{t}\in U, n\in\mathbb{N}}$.

In the following of the section we apply Theorem \ref{thm17} to various families of PIFS to deduce some interesting results. The first application is on the Hausdorff dimension of the attractors of a family of infinite PIFS. 

\begin{corollary}\label{cor4}
Let 
\begin{center}
$\big\{S^\bold{t}=\{s_i^{\bold{t}}: X\rightarrow X\}_{i\in \mathbb{N}}\in \Gamma_X(\theta)\big\}_{\bold{t}\in U}$
\end{center}
be a family of infinite parabolic iterated function systems satisfying the continuity and transversality condition with respect to the vector-time parameter $\bold{t}\in U$.  For an ergodic probability measure $\mu$ on the symbolic space $\mathbb{N}^\infty$ with positive entropy $h_\mu(\sigma)$ and $\mu(\cup_{n=1}^\infty \mathbb{N}_n^\infty)=1$, let $\mu_n$ be its $n$-th concentrating measure for any $n\in\mathbb{N}$. If the sequence of infinite random variables in $\bold{Y}$ under the law $\mu$ is independent, then 

\begin{enumerate}[(i).]
\item $HD(J_{\bold{t}})\geq\lim_{n\rightarrow\infty}\min\Big\{ \cfrac{h_{\mu_n}(\sigma)}{\lambda^{\bold{t}}_{\mu_n}(\sigma)},1\Big\}$ for Lebesgue \emph{a.e.} $\bold{t}\in U$.

\item $\mathfrak{L}^1(J_{\bold{t}})>0$ for Lebesgue \emph{a.e.} $\bold{t}\in \Big\{\bold{t}: \limsup_{n\rightarrow\infty}\big\{\cfrac{h_{\mu_n}(\sigma)}{\lambda^{\bold{t}}_{\mu_n}(\sigma)}\big\}>1\Big\}$.
\end{enumerate}
\end{corollary}
\begin{proof}
These results follow directly from Theorem \ref{thm17} and the fact that $supp(\nu_{\bold{t}})\subset J_{\bold{t}}$ for any $\bold{t}\in U$.
\end{proof} 
 
More interesting corollaries on the Hausdorff dimension of the attractors can be formulated from Theorem \ref{thm17}, considering different limit behaviours of the sequence $\Big\{\cfrac{h_{\mu_n}(\sigma)}{\lambda^{\bold{t}}_{\mu_n}(\sigma)}\Big\}_{\bold{t}\in U, n\in\mathbb{N}}$ and the fact that $supp(\nu_{\bold{t}})$ is always contained in $J_{\bold{t}}$.
 
Now we apply Theorem \ref{thm17} or Corollary \ref{cor4} to some families of infinite PIFS with \emph{exploding} measures $\mu$ on $\mathbb{N}^\infty$, that is, measures $\mu$ satisfying 
\begin{center}
$h_\mu(\sigma)=\infty$.
\end{center} 
There will be some stronger results available in due courses.

\begin{corollary}\label{thm16}
Let 
\begin{center}
$\big\{S^\bold{t}=\{s_i^{\bold{t}}: X\rightarrow X\}_{i\in \mathbb{N}}\in \Gamma_X(\theta, V, \gamma, u, M)\big\}_{\bold{t}\in U}$
\end{center}
be a family of infinite parabolic iterated function systems satisfying the continuity and transversality condition with respect to the vector-time parameter $\bold{t}\in U$ for some fixed open neighbourhood $V$ of $v$ and some fixed parameters $0<\theta,\gamma, u<1, M>0$. For an ergodic exploding probability measure $\mu$ on the symbolic space $\mathbb{N}^\infty$ with $\mu(\cup_{n=1}^\infty \mathbb{N}_n^\infty)=1$, if the sequence of infinite random variables in $\bold{Y}$ under the law $\mu$ is independent, then 
\begin{enumerate}[(i).]
\item For Lebesgue \emph{a.e.} $\bold{t}\in U$, $dim_* \nu_\bold{t}=dim^* \nu_\bold{t}=1$,

\item $\nu_\bold{t}$ is absolutely continuous for Lebesgue \emph{a.e.} $\bold{t}\in U$,
\end{enumerate}
in which $\nu_\bold{t}=\mu\circ\pi_\bold{t}^{-1}$ is the  projective measure at time $\bold{t}$.
\end{corollary}   

\begin{proof}
Note that since 
\begin{equation}\label{eq40}
\inf_{n\in \mathbb{N}, x\in X}\{|(s_n^{\bold{t}})'(x)|\}\geq u
\end{equation}
for any $\bold{t}\in U$, we have
\begin{equation}\label{eq41}
\lambda^\bold{t}_\mu(\sigma)\leq -\log u.
\end{equation}
Apply Theorem \ref{thm17} to the projective measures $\{\nu_\bold{t}\}_{\bold{t}\in U}$ originated from the family of infinite parabolic iterated function systems 
\begin{center}
$\big\{S^\bold{t}=\{s_i^{\bold{t}}: X\rightarrow X\}_{i\in \mathbb{N}}\in \Gamma_X(\theta, V, \gamma, u, M)\big\}_{\bold{t}\in U}$
\end{center}
satisfying the continuity and transversality condition (note that $\Gamma_X(\theta, V, \gamma, u, M)\subset \Gamma_X(\theta)$), the two conclusions follow instantly in virtue of (\ref{eq41}), $h_\mu(\sigma)=\infty$ and Lemma \ref{lem12}.

\end{proof}

\section{Dimensional estimates of the  exceptional parameters}

This section is devoted to estimation on the upper bound of the local (global) Hausdorff dimension of the  exceptional parameters. We mean to prove Theorem \ref{thm12} here. We start by showing several preceding results on the limit behaviours of the (families of truncated) finite sub-systems of an (family of) infinite PIFS.

\begin{lemma}\label{lem14}
Let 
\begin{center}
$S=\{s_i: X\rightarrow X\}_{i\in \mathbb{N}}\in\Gamma_X(\theta)$
\end{center}
be a PIFS.  Let $\mu$ be a finite measure on $\mathbb{N}^\infty$ with its concentrating measures $\{\mu_n\}_{n\in\mathbb{N}}$ on $\{\mathbb{N}_n^\infty\}_{n\in\mathbb{N}}$ respectively. Considering the sequence of Lyapunov exponents  $\{\lambda_{\mu_n}(\sigma)\}_{n=1}^\infty$ of the truncated systems $\big\{S_n=\{s_i: X\rightarrow X\}_{i\in \mathbb{N}_n}\big\}_{n\in\mathbb{N}}$, we have
\begin{center}
$\lim_{n\rightarrow\infty} \lambda_{\mu_n}(\sigma)=\lambda_{\mu}(\sigma)$.
\end{center}
\end{lemma}
\begin{proof}
First we represent $\lambda_{\mu}(\sigma)$ as
\begin{center}
$\lambda_{\mu}(\sigma)=\sum_{i=1}^\infty -\int_{[i]}\log|s_i'(\pi\circ\sigma(\omega))|d\mu(\omega)$.
\end{center}
Note that 
\begin{equation}\label{eq39}
-\int_{[i]}\log|s_i'(\pi\circ\sigma(\omega))|d\mu(\omega)<\infty
\end{equation}
for any $1\leq i<\infty$.
We distinguish the case $\lambda_{\mu}(\sigma)=\infty$ from the case $\lambda_{\mu}(\sigma)<\infty$. If $\lambda_{\mu}(\sigma)=\infty$, we will claim that for any $a>0$, there exists an integer $N_a$ large enough, such that for any $n>N_a$, we have
\begin{center}
$\lambda_{\mu_n}(\sigma)>a$.
\end{center}
This means $\lim_{n\rightarrow\infty} \lambda_{\mu_n}(\sigma)=\infty$. To see this, since $\lambda_{\mu}(\sigma)=\infty$, for any $a>0$, there exists $N_*$ large enough, such that 
\begin{center}
$\sum_{i=1}^{N_*} -\int_{[i]}\log|s_i'(\pi\circ\sigma(\omega))|d\mu(\omega)>2a$.
\end{center}
Now choose some small $0<\epsilon<a$, due to \ref{eq39}, for any $1\leq i\leq N_*$, we can find an integer $N_i$ large enough, such that
\begin{center}
$a_i:=\big|\int_{[i]}\log|s_i'(\pi\circ\sigma(\omega))|d\mu(\omega)-\int_{[i]}\log|s_i'(\pi\circ\sigma(\omega))|d\mu_n(\omega)\big|<\cfrac{\epsilon}{2^i}$
\end{center}
for any $n>N_i$. Now let $N_a=\max\{N_i\}_{i=1}^{N_*}$, we have
\begin{center}
$
\begin{array}{ll}
&\sum_{i=1}^{N_*} -\int_{[i]}\log|s_i'(\pi\circ\sigma(\omega))|d\mu_n(\omega)\\
\geq & \sum_{i=1}^{N_*} -\int_{[i]}\log|s_i'(\pi\circ\sigma(\omega))|d\mu(\omega)-\sum_{i=1}^{N_*} a_i\\
\geq & 2a-\sum_{i=1}^{N_*} \cfrac{\epsilon}{2^i}\\
> & a
\end{array}
$
\end{center}
for any $n>N_a$, which justifies the claim. 

The case $\lambda_{\mu}(\sigma)<\infty$ can be dealt with in a similar way, which is left to the interested readers.
\end{proof}

Similar to Remark \ref{rem1}, lemma \ref{lem14} holds for any sequence of measures converging to the measure $\mu$ under the setwise topology on $\mathcal{\hat{M}}(X)$, not only for the sequence of its concentrating measures, for any topological ambient space $X$. 

\begin{lemma}\label{lem13}
Let $G\subset U$ be a subset. Consider a sequences of real functions $\{f_n: G\rightarrow\mathbb{R}\}_{n=1}^\infty$ such that 
\begin{center}
$f_n(\bold{t})\rightarrow f(\bold{t})$ 
\end{center}
as $n\rightarrow\infty$ for some $f(\bold{t}): G\rightarrow\mathbb{R}$ at any $\bold{t}\in G$. Now if
\begin{center}
$HD(\{\bold{t}\in G: f_n(\bold{t})< 0\})\leq h_n$,
\end{center}
for a sequence of reals $\{h_n\}_{n=1}^\infty$, then we have
\begin{center}
$HD(\{\bold{t}\in G: f(\bold{t})< 0\})\leq \limsup_{n\rightarrow\infty} h_n$.
\end{center}
\end{lemma}
\begin{proof}
Note that since $\lim_{n\rightarrow\infty} f_n(\bold{t})\rightarrow f(\bold{t})$ at any time  $\bold{t}\in G$, then
\begin{center}
$\{\bold{t}\in G: f(\bold{t})< 0\}\subset \cap_{k=1}^\infty\cup_{n=k}^\infty\{\bold{t}\in G: f_n(\bold{t})< 0\}$.
\end{center}
This means
\begin{center}
$
\begin{array}{ll}
& HD(\{\bold{t}\in G: f(\bold{t})< 0\})\\
\leq & \inf\big\{HD(\cup_{n=k}^\infty\{\bold{t}\in G: f_n(\bold{t})<0\})\big\}_{k=1}^\infty\\
\leq & \inf\big\{\sup \{h_n\}_{n=k}^\infty\big\}_{k=1}^\infty\\
=&\limsup_{n\rightarrow\infty} h_n.
\end{array}
$
\end{center}
\end{proof}
Be careful that Lemma \ref{lem13} will not be true if the two $<$ are both substituted by $\leq$ in it. Now we are well prepared to prove Theorem \ref{thm12}. \\

Proof of Theorem \ref{thm12}:\\

\begin{proof}
Since $\mu$ is ergodic and $\bold{Y}$ is independent under the law $\mu$,  according to Lemma \ref{lem16}, consider the sequence of measures $\{\nu_{n,\bold{t}}\}_{\bold{t}\in U, n\in\mathbb{N}}$ projected from the ergodic concentrating measures $\{\mu_n\}_{n\in\mathbb{N}}$ of $\mu$. For fixed $n\in\mathbb{N}$, as the family of $n$-th truncated finite PIFS
\begin{center}
$\big\{S_n^\bold{t}=\{s_i^{\bold{t}}: X\rightarrow X\}_{i\in \mathbb{N}_n}\in \Gamma_X(\theta, V_n, \gamma_n, u_n, M_n)\big\}_{\bold{t}\in U}$
\end{center}
satisfying the continuity and strong transversality condition  with respect to the time parameter $\bold{t}\in U$ for some open neighbourhood $V_n$ of $v$, some $0<\gamma_n, u_n<1$ and $M_n>0$, apply \cite[Theorem 5.3]{SSU1} to the  family of truncated finite PIFS $S_n^\bold{t}$, we have
\begin{equation}\label{eq49}
HD\Big(\big\{\bold{t}\in G: dim^* \nu_{n,\bold{t}}<\min\{\frac{h_{\mu_n}(\sigma)}{\lambda_{\mu_n}^{\bold{t}}(\sigma)},\alpha\}\big\}\Big)\leq \min\Big\{\sup_{\bold{t}\in G} \frac{h_{\mu_n}(\sigma)}{\lambda_{\mu_n}^{\bold{t}}(\sigma)}, \alpha\Big\}+d-1
\end{equation}
for any $n\in\mathbb{N}$ and $\epsilon>0$. 
Note that 
\begin{center}
$dim^* \nu_{\bold{t}}= \lim_{n\rightarrow\infty} dim^* \nu_{n,\bold{t}}$
\end{center}
since $\lim_{n\rightarrow\infty}\nu_{n,\bold{t}}\stackrel{s}{\rightarrow}\nu_{\bold{t}}$ and
\begin{center}
$\lim_{n\rightarrow\infty} \min\{\cfrac{h_{\mu_n}(\sigma)}{\lambda_{\mu_n}^{\bold{t}}(\sigma)},\alpha\}$
\end{center}
exists for any $\bold{t}\in G$ since $\lim_{n\rightarrow\infty} \cfrac{h_{\mu_n}(\sigma)}{\lambda_{\mu_n}^{\bold{t}}(\sigma)}$ exists everywhere on $G$. Moreover, 
\begin{center}
$\lim_{n\rightarrow\infty} \Big(\min\big\{\sup_G \cfrac{h_{\mu_n}(\sigma)}{\lambda_{\mu_n}^{\bold{t}}(\sigma)}, \alpha\big\}+d-1\Big)=K_{\alpha, G}$.
\end{center}
Apply Lemma \ref{lem13} here, we get (\ref{eq38}).

\end{proof}

In the non-exploding case, Theorem \ref{thm12} degenerates into a more simple version as following. 

\begin{corollary}
Let 
\begin{center}
$\big\{S^\bold{t}=\{s_i^{\bold{t}}: X\rightarrow X\}_{i\in \mathbb{N}}\in \Gamma_X(\theta)\big\}_{\bold{t}\in U}$
\end{center}
be a family of infinite parabolic iterated function systems satisfying the continuity and strong transversality condition with respect to the vector-time parameter $\bold{t}\in U$. For an ergodic probability measure $\mu$ on the symbolic space $\mathbb{N}^\infty$ satisfying $\mu(\cup_{n=1}^\infty \mathbb{N}_n^\infty)=1$ and $0<h_\mu(\sigma), \lambda^{\bold{t}}_\mu(\sigma)<\infty$ at any time $\bold{t}\in U$, if the sequence of infinite random variables in $\bold{Y}$ under the law $\mu$ is independent, then 
\begin{equation}\label{eq38}
HD\Big(\Big\{\bold{t}\in G: dim^* \nu_\bold{t}<\min\big\{ \cfrac{h_{\mu}(\sigma)}{\lambda^{\bold{t}}_{\mu}(\sigma)},\alpha\big\}\Big\}\Big)\leq \min\Big\{\sup_{\bold{t}\in G} \cfrac{h_{\mu}(\sigma)}{\lambda^{\bold{t}}_{\mu}(\sigma)}, \alpha\Big\}+d-1
\end{equation}
for any $0<\alpha<1$ and $G\subset U$.
\end{corollary}
\begin{proof}
Under the assumption $0<h_\mu(\sigma), \lambda^{\bold{t}}_\mu(\sigma)<\infty$ at any time $\bold{t}\in U$, considering Lemma \ref{lem12} and Lemma \ref{lem14}, the result follows instantly from  Theorem \ref{thm12}. 

\end{proof}

In the exploding case, that is, either $h_\mu(\sigma)=\infty$ or $\lambda^{\bold{t}}_\mu(\sigma)=\infty$ for some $\bold{t}\in U$, the result depends heavily on the asymptotic behaviour of the sequence of concentrating measures of $\mu$, however, some estimations on the Hausdorff dimension of some form of exceptional parameters may still be possible upon reforming Lemma \ref{lem13}.

Note that most of the notions and results in this work apply to families of \emph{hyperbolic iterated function systems}, which are families of iterated function systems constituted by contractive hyperbolic maps only.

\end{document}